\documentclass[12pt,a4paper]{article}
 
\usepackage{amsmath}
\usepackage{amssymb}
\usepackage{amsthm}
\usepackage{enumitem}
\usepackage{float}
\usepackage{graphicx}
\usepackage{hyperref,url}
\hypersetup{
	colorlinks=true,
  linkcolor=red,          
  citecolor=blue,         
  filecolor=magenta,      
  urlcolor=cyan           
}

\newtheorem{prop}{Proposition}[section]
\newtheorem{lemma}[prop]{Lemma}
\newtheorem{theorem}[prop]{Theorem}

\theoremstyle{definition}

\newtheorem{remark}[prop]{Remark}

\newcommand{\N}{\mathbb{N}}

\newcommand{\seqnum}[1]{\href{https://oeis.org/#1}{\rm \underline{#1}}}
\newcommand{\mylabel}[2]{#2\def\@currentlabel{#2}\label{#1}}

\begin{document}

\title{Equivalent conditions for the $n$th element of the Beatty sequence $B_{\sqrt{2}}$ being even}

\author{Sela Fried\thanks{The author is a teaching fellow in the Department of Computer Science at the Israel Academic College in Ramat Gan.} 
\\ \href{mailto:friedsela@gmail.com}{friedsela@gmail.com}}
\date{} 
\maketitle

\begin{abstract}
We provide equivalent conditions for the $n$th element of the Beatty sequence $B_{\sqrt{2}}$ being even. In particular, we show that the integer sequences \seqnum{A090892} and \seqnum{A120752} in the OEIS are essentially identical. 
\end{abstract} 
	
\section{Introduction}

A Beatty sequence is the sequence of integers obtained by taking the floor of the positive multiples of a positive irrational number, i.e., if $r>0$ is irrational, then the corresponding Beatty sequence $B_r$ is given by $$B_r = \left(\left\lfloor nr \right\rfloor\right)_{n\in\N}.$$ Beatty sequences are named after Samuel Beatty who brought them to the attention of the mathematical community by posing a problem in the American Mathematical Monthly \cite{Beat}, in which the readers of the journal were asked to prove that, if $r>1$ is irrational and $s=r/(r-1)$, then $B_r$ and $B_s$ are complementary sequences, i.e., every natural number belongs to exactly one of the two sequences. 

To the best of our knowledge, little attention has been given to the parities of the elements of Beatty sequences. That these might prove interesting is perhaps suggested by the image below, that visualizes the first $10^5$ elements of the parity sequence of $B_{\sqrt{2}}$ (cf.\ \seqnum{A083035} in the On-Line Encyclopedia of Integer Sequences (OEIS) \cite{OL}). The visualization  
relies on a simple but ingenious method capable of visualizing any binary sequence $(a_n)_{n\in\N}$: Start at $(0,0)$, facing in the direction of the vector $(1,0)$. Now, for every $n\in\N$, execute the following two actions: (a) Turn right if $a_n = 0$ and left otherwise. (b) Proceed forward one step of unit length.  

The result of this process is a walk in the plane that we refer to as the \emph{Cloitre walk of $(a_n)_{n\in\N}$}, since Benoit Cloitre seems to be the first to use this method.

\begin{figure}[H]
\centering{\includegraphics[width=1.00\textwidth]{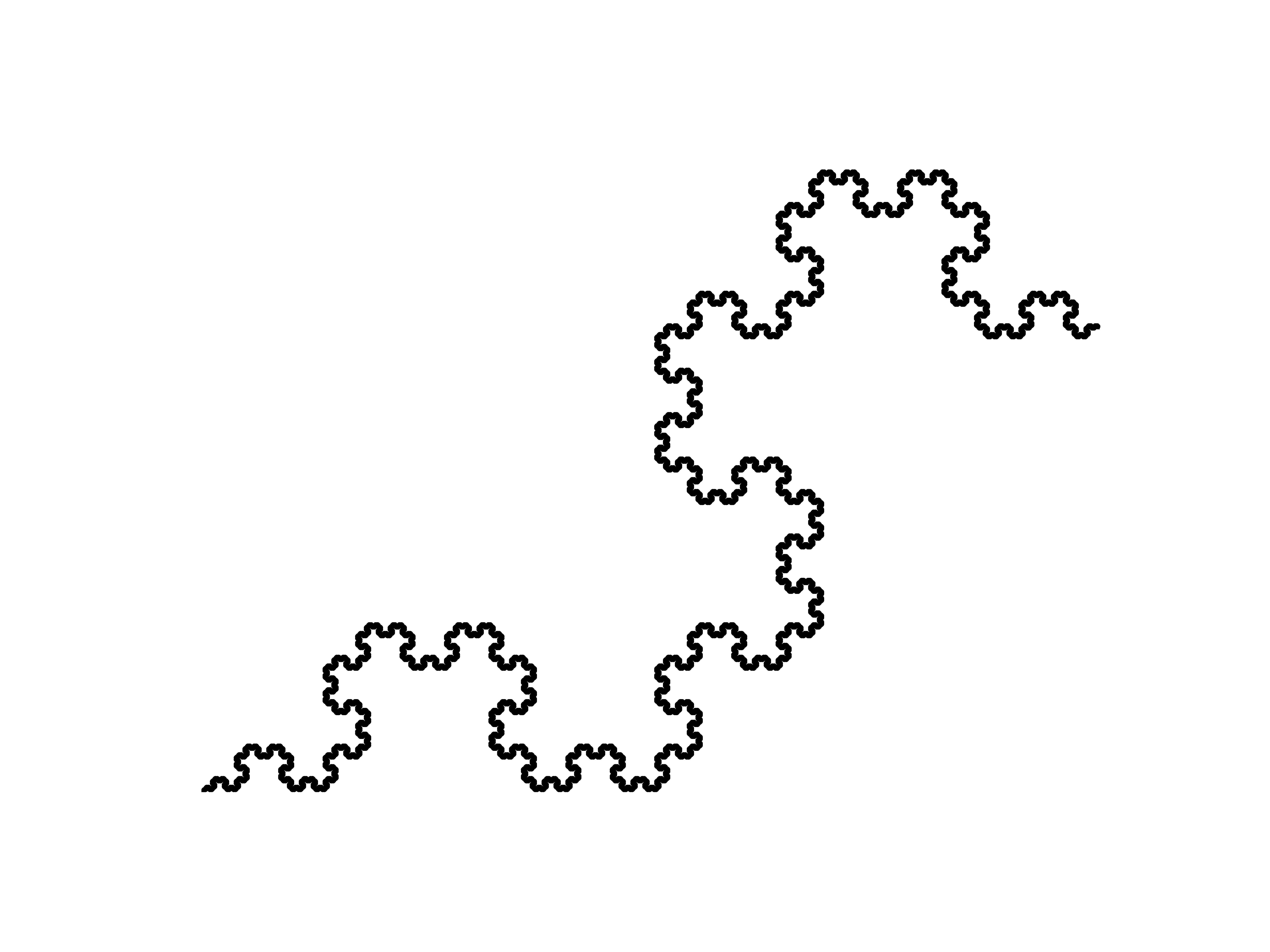}}
\caption{The Cloitre walk of \seqnum{A083035} that corresponds to the sequence of parities of $B_{\sqrt{2}}$.}
\label{fig;eg}
\end{figure}

In this short note we provide several equivalent conditions for the $n$th element of the Beatty sequence $B_{\sqrt{2}}$ being even. Along the way we prove that \seqnum{A090892} and \seqnum{A120752} are essentially identical since deleting the first two elements of the former sequence results in the latter (cf.\ \cite{S}).

\section{Main results}

We denote by $\N$ the set of natural numbers and by $\{x\}$ the fractional part of a real number $x$. The purpose of this note is to prove the following theorem. 

\begin{theorem}\label{thm;1}
    Let $2\leq n\in\N$. The following conditions are equivalent:
    \begin{enumerate}
        \item [(a)]\label{a}$\left\lfloor \sqrt{2}n\right\rfloor$ is even, i.e., the $n$th element of the Beatty sequence $B_{\sqrt{2}}$ is even.
        \item [(b)]\label{b} $\left\{\frac{n}{\sqrt{2}}\right\}\leq\frac{1}{2}$.
        \item [(c)]\label{c} $\left\{\frac{n}{\sqrt{2}}\right\}<\frac{1}{2}$.
        \item [(d)]\label{d} $\left\lfloor \sqrt{2}n\left\lfloor \frac{n}{\sqrt{2}}\right\rfloor \right\rfloor =\left\lfloor \frac{n}{\sqrt{2}}\left\lfloor \sqrt{2}n\right\rfloor \right\rfloor$.
        \item [(e)]\label{e} $\left\lfloor \frac{n}{\sqrt{2}}\right\rfloor =\left\lfloor \sqrt{n^{2}-\left\lfloor \frac{n}{\sqrt{2}}\right\rfloor ^{2}}\right\rfloor$.
        \item [(f)]\label{f} $\left\{\frac{n}{\sqrt{2}}\right\}< \sqrt{\left\lfloor \frac{n}{\sqrt{2}}\right\rfloor ^{2}+\left\lfloor \frac{n}{\sqrt{2}}\right\rfloor +\frac{1}{2}}-\left\lfloor \frac{n}{\sqrt{2}}\right\rfloor$.        
    \end{enumerate}
\end{theorem}

\begin{remark}
Conditions (\hyperref[b]{b}) and (\hyperref[d]{d}) correspond to the definitions of \seqnum{A120752} and \seqnum{A090892}, respectively. Cloitre made a comment to \seqnum{A090892} stating that conditions (\hyperref[a]{a}) and (\hyperref[d]{d}) are equivalent. 
\end{remark}

We shall make use of the following inequality.

\begin{lemma}\label{lem;1}
Let $x$ be a nonnegative real number. Then $$\sqrt{x^{2}+x+\frac{1}{2}}-x>\frac{1}{2}.$$ 
\end{lemma}

\begin{proof}
We have
\begin{align}
\sqrt{x^{2}+x+\frac{1}{2}}-x	&=\frac{x+\frac{1}{2}}{\sqrt{x^{2}+x+\frac{1}{2}}+x}\nonumber\\
	&>\frac{x+\frac{1}{2}}{\sqrt{x^{2}+\sqrt{2}x+\frac{1}{2}}+x}\nonumber\\
	&=\frac{x+\frac{1}{2}}{2x+\frac{1}{\sqrt{2}}}\nonumber\\
	&=\frac{x+\frac{1}{2}}{2\left(x+\frac{1}{\sqrt{2}}\frac{1}{2}\right)}\nonumber\\
	&>\frac{x+\frac{1}{2}}{2\left(x+\frac{1}{2}\right)}=\frac{1}{2}.  \nonumber\qedhere  
\end{align}
\end{proof}

\begin{proof}[Proof of Theorem \ref{thm;1}]
``(\hyperref[e]{e}) $\iff$ (\hyperref[f]{f})": We have 
\begin{align}
\left\lfloor \frac{n}{\sqrt{2}}\right\rfloor =\left\lfloor \sqrt{n^{2}-\left\lfloor \frac{n}{\sqrt{2}}\right\rfloor ^{2}}\right\rfloor&\iff \left\lfloor \frac{n}{\sqrt{2}}\right\rfloor \leq\sqrt{n^{2}-\left\lfloor \frac{n}{\sqrt{2}}\right\rfloor ^{2}}<\left\lfloor \frac{n}{\sqrt{2}}\right\rfloor +1\nonumber\\
&\iff 2\left\lfloor \frac{n}{\sqrt{2}}\right\rfloor^2 \leq n^{2}<2\left\lfloor \frac{n}{\sqrt{2}} \right\rfloor^2 +2\left\lfloor \frac{n}{\sqrt{2}} \right\rfloor+1.\label{eq;11}
\end{align}
Now, since $$2\left\lfloor \frac{n}{\sqrt{2}}\right\rfloor^2 \leq2\left(\frac{n}{\sqrt{2}}\right)^2=n^2,$$ we have 
\begin{align}
(\ref{eq;11})&\iff n^{2}<2\left\lfloor \frac{n}{\sqrt{2}} \right\rfloor^2 +2\left\lfloor \frac{n}{\sqrt{2}} \right\rfloor+1.\label{eq;22}
\end{align}
Let us denote $\left\{\frac{n}{\sqrt{2}}\right\}$ by $\sigma$. Thus, $\sigma\in[0,1)$ and $\left\lfloor \frac{n}{\sqrt{2}}\right\rfloor =\frac{n}{\sqrt{2}}-\sigma$. Then 
\begin{align}
(\ref{eq;22})&\iff n^{2}<2\left(\frac{n}{\sqrt{2}}-\sigma\right)^{2}+2\left(\frac{n}{\sqrt{2}}-\sigma\right)+1\nonumber\\&\iff \sqrt{2}n(2\sigma-1)<\overbrace{2\sigma^{2}-2\sigma+1}^{>0 \textnormal{ for every }\sigma}\nonumber \\ &\iff\left(\sigma \leq \frac{1}{2}\right) \textnormal{ or } \left(\sigma > \frac{1}{2} \textnormal{ and } n<\frac{2\sigma^{2}-2\sigma+1}{\sqrt{2}(2\sigma-1)}\right).\label{eq;33}
\end{align}
Now, 
\begin{align}
\sigma > \frac{1}{2} \textnormal{ and } n<\frac{2\sigma^{2}-2\sigma+1}{\sqrt{2}(2\sigma-1)}&\iff \sigma > \frac{1}{2} \textnormal{ and } \left\lfloor \frac{n}{\sqrt{2}}\right\rfloor <\frac{1-2\sigma^{2}}{2(2\sigma-1)}\nonumber\\&\iff \sigma > \frac{1}{2} \textnormal{ and } 2\sigma^{2}+4\left\lfloor \frac{n}{\sqrt{2}}\right\rfloor \sigma-1-2\left\lfloor \frac{n}{\sqrt{2}}\right\rfloor <0\nonumber\\&\iff \frac{1}{2}<\sigma< \overbrace{\sqrt{\left\lfloor \frac{n}{\sqrt{2}}\right\rfloor ^{2}+\left\lfloor \frac{n}{\sqrt{2}}\right\rfloor +\frac{1}{2}}-\left\lfloor \frac{n}{\sqrt{2}}\right\rfloor }^{>\frac{1}{2}, \textnormal{ by Lemma \ref{lem;1}}}.\nonumber
\end{align}
Thus,
$$(\ref{eq;33}) \iff \sigma<\sqrt{\left\lfloor \frac{n}{\sqrt{2}}\right\rfloor ^{2}+\left\lfloor \frac{n}{\sqrt{2}}\right\rfloor +\frac{1}{2}}-\left\lfloor \frac{n}{\sqrt{2}}\right\rfloor.$$

``(\hyperref[b]{b}) $\iff$ (\hyperref[c]{c})": This is clear.

``(\hyperref[c]{c}) $\iff$ (\hyperref[d]{d})": Let us denote $\left\{\sqrt{2}n\left\lfloor \frac{n}{\sqrt{2}}\right\rfloor \right\}$ by $\sigma$. Thus, $$\left\lfloor \sqrt{2}n\left\lfloor \frac{n}{\sqrt{2}}\right\rfloor \right\rfloor =\sqrt{2}n\left\lfloor \frac{n}{\sqrt{2}}\right\rfloor -\sigma.$$ Then,
\begin{align}
&\left\lfloor \sqrt{2}n\left\lfloor \frac{n}{\sqrt{2}}\right\rfloor \right\rfloor =\left\lfloor \frac{n}{\sqrt{2}}\left\lfloor \sqrt{2}n\right\rfloor \right\rfloor\iff \nonumber\\&
\left\lfloor \sqrt{2}n\left\lfloor \frac{n}{\sqrt{2}}\right\rfloor \right\rfloor \leq\frac{n}{\sqrt{2}}\left\lfloor \sqrt{2}n\right\rfloor <\left\lfloor \sqrt{2}n\left\lfloor \frac{n}{\sqrt{2}}\right\rfloor \right\rfloor +1
\iff\nonumber\\
&n\left(2\left\lfloor \frac{n}{\sqrt{2}}\right\rfloor -\left\lfloor \sqrt{2}n\right\rfloor \right)\leq\sqrt{2}\sigma \;\text{ and }\;
\frac{n}{\sqrt{2}}\left(\left\lfloor \sqrt{2}n\right\rfloor -2\left\lfloor \frac{n}{\sqrt{2}}\right\rfloor \right)<1-\sigma.\label{eq;66}
\end{align}
It is easy to see that, for every nonnegative real number $x$ and $m\in\N$, we have $$\left\lfloor x\right\rfloor -m\left\lfloor \frac{x}{m}\right\rfloor \in\left\{ 0,1,\ldots,m-1\right\}.$$ In particular, \begin{equation}\label{eq;1}\left\lfloor \sqrt{2}n\right\rfloor -2\left\lfloor \frac{n}{\sqrt{2}}\right\rfloor\in\{0,1\}.\end{equation} It follows that the first inequality in (\ref{eq;66}) holds trivially. Furthermore, since $n\geq 2$, if $\left\lfloor \sqrt{2}n\right\rfloor -2\left\lfloor \frac{n}{\sqrt{2}}\right\rfloor=1$, then the second inequality in (\ref{eq;66}) cannot hold. We conclude that \begin{align}
(\ref{eq;66}) & \iff \left\lfloor \sqrt{2}n\right\rfloor =2\left\lfloor \frac{n}{\sqrt{2}}\right\rfloor\nonumber\\&\iff 2\left\lfloor \frac{n}{\sqrt{2}}\right\rfloor \leq\sqrt{2}n<2\left\lfloor \frac{n}{\sqrt{2}}\right\rfloor +1.\label{eq;88}
\end{align}
The inequality $2\left\lfloor \frac{n}{\sqrt{2}}\right\rfloor \leq\sqrt{2}n$ holds trivially. Thus, 
\begin{align}
(\ref{eq;88}) & \iff \sqrt{2}n<2\left\lfloor \frac{n}{\sqrt{2}}\right\rfloor +1.\nonumber\\&\iff\frac{n}{\sqrt{2}}<\left\lfloor \frac{n}{\sqrt{2}}\right\rfloor +\frac{1}{2}\nonumber\\&\iff\left\{\frac{n}{\sqrt{2}}\right\}<\frac{1}{2}.\nonumber
\end{align}

``(\hyperref[c]{c}) $\iff$ (\hyperref[f]{f})": 
By Lemma \ref{lem;1}, $$\sqrt{\left\lfloor \frac{n}{\sqrt{2}}\right\rfloor ^{2}+\left\lfloor \frac{n}{\sqrt{2}}\right\rfloor +\frac{1}{2}}-\left\lfloor \frac{n}{\sqrt{2}}\right\rfloor >\frac{1}{2}$$ and that settles the ``$\Longrightarrow$" implication. For the other implication, we need to show that there exists no $2\leq n\in\N$ such that 
\begin{equation}\label{eq;1122}
\frac{1}{2}\leq\left\{\frac{n}{\sqrt{2}}\right\} <\sqrt{\left\lfloor \frac{n}{\sqrt{2}}\right\rfloor ^{2}+\left\lfloor \frac{n}{\sqrt{2}}\right\rfloor +\frac{1}{2}}-\left\lfloor \frac{n}{\sqrt{2}}\right\rfloor.
\end{equation}
Setting $\sigma=\left\{\frac{n}{\sqrt{2}}\right\}$, we have 
\begin{align}
(\ref{eq;1122})&\iff\frac{1}{2}\leq\sigma<\sqrt{\left(\frac{n}{\sqrt{2}}-\sigma\right)^{2}+\frac{n}{\sqrt{2}}-\sigma+\frac{1}{2}}-\left(\frac{n}{\sqrt{2}}-\sigma\right)\nonumber\\&\iff\frac{1}{2}+\left(\frac{n}{\sqrt{2}}-\sigma\right)\leq\sigma+\left(\frac{n}{\sqrt{2}}-\sigma\right)<\sqrt{\left(\frac{n}{\sqrt{2}}-\sigma\right)^{2}+\frac{n}{\sqrt{2}}-\sigma+\frac{1}{2}}\nonumber\\&\iff \left(\frac{n}{\sqrt{2}}+\frac{1}{2}-\sigma\right)^{2}\leq\frac{n^{2}}{2}<\left(\frac{n}{\sqrt{2}}-\sigma\right)^{2}+\frac{n}{\sqrt{2}}-\sigma+\frac{1}{2}\nonumber \\&\iff0<\sigma^{2}-\sqrt{2}n\sigma+\frac{n}{\sqrt{2}}-\sigma+\frac{1}{2}\leq\frac{1}{4}\nonumber\\&\iff 0<-\left(\frac{n}{\sqrt{2}}-\left\lfloor \frac{n}{\sqrt{2}}\right\rfloor \right)\left(\frac{n}{\sqrt{2}}+\left\lfloor \frac{n}{\sqrt{2}}\right\rfloor\right)+\left\lfloor \frac{n}{\sqrt{2}}\right\rfloor +\frac{1}{2}\leq\frac{1}{4}\nonumber\\&\iff 0<\left\lfloor \frac{n}{\sqrt{2}}\right\rfloor^2
+\left\lfloor \frac{n}{\sqrt{2}}\right\rfloor +\frac{1}{2}(1-n^2)\leq\frac{1}{4}.\label{eq;3573}
\end{align}
Now, both $\left\lfloor \frac{n}{\sqrt{2}}\right\rfloor^2
+\left\lfloor \frac{n}{\sqrt{2}}\right\rfloor$ and $1-n^2$ are integers. Thus, the two inequalities in (\ref{eq;3573}) cannot hold simultaneously. 

``(\hyperref[a]{a}) $\iff$ (\hyperref[c]{c})": We have $$\left\{ \frac{n}{\sqrt{2}}\right\} <\frac{1}{2}\iff\frac{n}{\sqrt{2}}-\left\lfloor \frac{n}{\sqrt{2}}\right\rfloor <\frac{1}{2}\iff\sqrt{2}n-2\left\lfloor \frac{n}{\sqrt{2}}\right\rfloor <1.$$ Since  $$\left\lfloor \sqrt{2}n\right\rfloor -2\left\lfloor \frac{n}{\sqrt{2}}\right\rfloor\leq \sqrt{2}n-2\left\lfloor \frac{n}{\sqrt{2}}\right\rfloor,$$
using $(\ref{eq;1})$, we conclude that \[ \sqrt{2}n-2\left\lfloor \frac{n}{\sqrt{2}}\right\rfloor <1\iff \left\lfloor \sqrt{2}n\right\rfloor -2\left\lfloor \frac{n}{\sqrt{2}}\right\rfloor=0\iff\left\lfloor \sqrt{2}n\right\rfloor \textnormal{ is even}.\qedhere\]
\end{proof}

\end{document}